\documentclass[12 pt]{amsart}

\usepackage{hyperref}
\usepackage{etex}
\usepackage[shortlabels]{enumitem}
\usepackage{amsmath}
\usepackage{amsxtra}
\usepackage{amscd}
\usepackage{amsthm}
\usepackage{adjustbox}
\usepackage{amsfonts}
\usepackage{amssymb}
\usepackage{eucal}
\usepackage[all]{xy}
\usepackage{graphicx}
\usepackage{tikz-cd}
\usepackage{mathrsfs}
\usepackage{subfiles}
\usepackage{mathpazo}
\usepackage[colorinlistoftodos, textsize=tiny]{todonotes}
\usepackage{morefloats}
\usepackage{pdfpages}
\usepackage{thm-restate}
\usepackage[percent]{overpic}
\usepackage[utf8]{inputenc}
\usepackage{epigraph}
\usepackage{csquotes}
\usepackage[margin=1in]{geometry}
\usepackage{adjustbox}
\usepackage{microtype}
\usepackage{verbatim}
\usepackage{stmaryrd}
\usepackage{scalerel}
\usepackage{stackengine}
\stackMath
\newcommand\reallywidehat[1]{%
\savestack{\tmpbox}{\stretchto{%
  \scaleto{%
    \scalerel*[\widthof{\ensuremath{#1}}]{\kern-.6pt\bigwedge\kern-.6pt}%
    {\rule[-\textheight/2]{1ex}{\textheight}}
  }{\textheight}%
}{0.5ex}}%
\stackon[1pt]{#1}{\tmpbox}%
}
\usepackage{mathtools}

\graphicspath{ {images/} }

\RequirePackage{color}
\definecolor{myred}{rgb}{0.75,0,0}
\definecolor{mygreen}{rgb}{0,0.5,0}
\definecolor{myblue}{rgb}{0,0,0.65}

\usepackage{color}
\newcommand{\daniel}[1]{{\color{blue} \sf
    $\spadesuit\spadesuit\spadesuit$ DANIEL: [#1]}}

\newcommand{\josh}[1]{{\color{red} \sf
    $\spadesuit\spadesuit\spadesuit$ JOSH: [#1]}}

\usepackage{hyperref}
\hypersetup{citecolor=blue}
\usepackage{tikz}
\usetikzlibrary{matrix,arrows,decorations.pathmorphing}

\theoremstyle{plain}
\newtheorem{theorem}[subsubsection]{Theorem}

\newtheorem{proposition}[subsubsection]{Proposition}
\newtheorem{lemma}[subsubsection]{Lemma}
\newtheorem{corollary}[subsubsection]{Corollary}

\theoremstyle{definition}
\newtheorem{definition}[subsubsection]{Definition}
\newtheorem{remark}[subsubsection]{Remark}
\newtheorem{example}[subsubsection]{Example}

\newtheorem{question}[subsubsection]{Question}
\newtheorem{conjecture}[subsubsection]{Conjecture}

\theoremstyle{remark}
\newtheorem{notation}[subsubsection]{Notation}

\numberwithin{equation}{section}
\newcommand\nc{\newcommand}
\nc\on{\operatorname}
\nc\renc{\renewcommand}

\DeclareMathOperator\Gr{Gr}
\DeclareMathOperator\Jac{Jac}
\DeclareMathOperator\GL{GL}
\DeclareMathOperator\SL{SL}
\DeclareMathOperator\End{End}

\nc\mf\mathfrak
\nc\mc\mathcal
\nc\mb\mathbb
\nc\msf\mathsf
\nc\mscr\mathscr

\def\fpbar{{\overline{\mb{F}}_p}}
\def\Spec{{\rm Spec}}
\newcommand{\defeq}{\vcentcolon=}
\newcommand{\Qlbar}{\overline{\mathbb{Q}}_{\ell}}

\title{Geometric local systems on the projective line minus four points}
\author{Yeuk Hay Joshua Lam and Daniel Litt}
\date{\today}

\begin{document}

\begin{abstract}
Let $J(m)$ be an $m\times m$ Jordan block with eigenvalue $1$. For $\lambda\in \mathbb{C}\setminus\{0,1\}$, we explicitly construct all rank $2$ local systems of geometric origin on $\mathbb{P}^1\setminus\{0,1,\lambda, \infty\}$, with local monodromy conjugate to $J(2)$ at $0,1,\lambda$ and conjugate to $-J(2)$ at $\infty$. The construction relies on Katz's middle convolution operation. We use our construction to prove two conjectures of Sun-Yang-Zuo (one of which was proven earlier by Lin-Sheng-Wang; the other was proven independently from us by Yang-Zuo).
\end{abstract}

\maketitle

\section{Introduction}\label{section:introduction}
It was known (in some sense) to Riemann \cite[Introduction]{katz1996rigid}, in his work on hypergeometric functions, that all rank two local systems on $\mathbb{P}^1\setminus\{0,1,\infty\}$ with trivial determinant and quasi-unipotent local monodromy are \emph{of geometric origin} --- that is, they arise in the cohomology of a family of varieties over $\mathbb{P}^1\setminus\{0,1,\infty\}$. The case of $\mathbb{P}^1$ minus $4$ points is more complicated; Beauville \cite{beauville1982familles} famously classified the $D\subset \mathbb{P}^1, |D|=4$ such that $\mathbb{P}^1\setminus D$ carries a family of elliptic curves with stable reduction along $D$, or equivalently, carries a rank $2$ $\mathbb{Z}$-local system of geometric origin with unipotent local monodromy. In general, rank $2$ motivic local systems on $\mathbb{P}^1$ minus four points are poorly understood (and the situation when one increases the rank or the number of points deleted is completely mysterious). 

The goal of this paper is to explicitly write down all local systems of geometric origin on $\mathbb{P}^1\setminus\{0,1,\infty,\lambda\}$, with local monodromies conjugate to $$\begin{pmatrix} 1 & 1 \\ 0 & 1\end{pmatrix}$$ at $0, 1,\lambda$ and with local monodromy conjugate to $$\begin{pmatrix} -1 & 1 \\ 0 & -1\end{pmatrix}$$ at $\infty$. In so doing we give a short proof of two conjectures of Sun-Yang-Zuo \cite[Conjectures 4.8 and 4.10]{syz} on motivic local systems on $X=\mathbb{P}^1\setminus\{0,1,\infty,\lambda\}$. \cite[Conjecture 4.8]{syz} was earlier proven in \cite{shengtorsion} by different methods. 

\begin{remark}
	While this paper was in preparation Yang and Zuo \cite{zuolong} independently claimed a fascinating proof of \cite[Conjecture 4.10]{syz},  via completely different techniques from ours\footnote{Our work began in December of 2022, whereas the work of loc.cit. had been ongoing for several years, as we were informed by the authors of \cite{zuolong}}. Their 128-page proof relies on $p$-adic Hodge theory and the Langlands correspondence; by contrast our proof is only a few pages and we \emph{explicitly} write down the relevant local systems and the families in whose cohomology they appear. 
\end{remark}
\subsection{Main results}
Let $X$ be a smooth complex curve, and in the case when $X$ is not proper, we pick a smooth  compactification $\overline{X}$ of $X$. Recall that a $\mb{C}$-local system $\mb{V}$ on $X$ is said to be of geometric origin if there exists a dense open $U\subset X$, and a  family of smooth proper algebraic varieties $\pi: \mc{Y} \to U$ such that $\mb{V}$ appears as a subquotient of $R^i\pi_*\mb{C}$ for some $i\geq 0$. (The appearance of $U$ in the definition is a slightly technical point which the reader may ignore at first pass.)

Motivated by the previous work of Faltings \cite{arakelovav}, Deligne showed that, for a fixed $X$, there are only finitely many $\mb{Q}$-local systems of fixed rank $n$ which are of geometric origin. One of our main motivations was whether a strengthened form of Deligne's theorem \cite{delignefinite} could hold:
\begin{question}\label{question}
Can there be infinitely many $\mb{C}$-local systems of rank $n$  on $X$ which are of geometric origin, and moreover whose  local monodromies at the boundary $\overline{X}-X$ are fixed?
\end{question}

\begin{remark}
The condition on local monodromies is certainly necessary for the above question to be interesting: indeed, the hypergeometric local systems of rank two on $X=\mb{P}^1-\{0, 1, \infty\}$ are all of geometric origin. However, the local monodromies at $0,1,\infty$ have eigenvalues being $m$-th roots of unity, and upon bounding $m$ there are only finitely many such hypergeometric local systems. We find the monodromy condition natural since it implies  that  we are considering points of the \emph{relative character varieties}, which are the natural replacement of character varieties in the case of non-proper $X$.

Note that \cite{litt2021arithmetic} shows that there are only finitely many $K$-local systems of geometric on $X$ if $K$ embeds in any finite extension of $\mathbb{Q}_p$, for any prime $p$.
\end{remark}

We now specialize to our case of interest. For any $\lambda \neq 0, 1$, let $X$ be  the curve $\mathbb{P}^1-\{0,1,\lambda, \infty\}$. We consider rank two local systems on $X$ satisfying the following condition $(\star)$: namely with local monodromy  conjugate to $$\begin{pmatrix} 1 & 1 \\ 0 & 1\end{pmatrix}$$ at $0, 1,\lambda$ and  local monodromy conjugate to $$\begin{pmatrix} -1 & 1 \\ 0 & -1\end{pmatrix}$$ at $\infty$. We recall the following:
\begin{definition}
A local system $\mb{V}$ on a manifold $Y$ is said to be MCG-finite if the set of isomorphism classes of local systems $\{f^*\mathbb{V}\}$ obtained by acting on $\mb{V}$ by elements $f$ of the mapping class group of $Y$ is finite. 
\end{definition}
We may now state our first result.
\begin{theorem}\label{thm:main1}
 For any $\lambda \neq 0,1$, there exists infinitely many   rank two local systems of geometric origin on $X$ satisfying condition $(\star)$. Moreover, any rank two local system of geometric origin  on $X$ satisfying $(\star)$ is MCG-finite. 
\end{theorem}
\begin{remark}
\autoref{thm:main1} answers \autoref{question} in the positive. It would be very interesting to further investigate \autoref{question} in the case of $X$ proper and to understand whether there is a qualitative difference between the two cases.
\end{remark}
This result was claimed independently, though inexplicitly, by Yang-Zuo \cite{zuolong}; they rephrase MCG-finiteness in terms of algebraic solutions to the Painlev\'e VI equation. In that language, these solutions were, we believe, discovered originally by Hitchin \cite{hitchin1995poncelet}; Lysovyy-Tykhyy \cite{lisovtyk}, in their classification of algebraic solutions to Painlev\'e VI, refer to them as the ``Cayley solutions". The reason for this name is that, under  condition $(\star)$, the  equation of the relative character variety is the \emph{Cayley cubic}: see \cite[Equation (61)]{lisovtyk}.

In fact, we classify all rank two local systems of geometric origin and satisfying $(\star)$. The following is the explicit form of our classification:
\begin{theorem}\label{thm:explicit-form}
Let $f: E\to \mathbb{P}^1$ be the $2:1$ cover branched over $\{0,1,\infty,\lambda\}$. Consider the fiber square 
$$\xymatrix{
Z'_m\ar[d] \ar[r] & W' \ar[d] \\
E \times \mathbb{P}^1  \ar[r]^{[m]\times \on{id}} & E \times \mathbb{P}^1 
}$$	
where $W'$ is the double cover of $E\times \mathbb{P}^1$ branched along the graph of $f$ and $E\times \{\infty\}$.
For any rank two local system  $\mathbb{V}$ of geometric origin on $X$ satisfying $(\star)$, $\mathbb{V}$ appears as a subquotient of the cohomology of the family $$Z'_m\to E\times \mathbb{P}^1\to \mathbb{P}^1$$ for some $m$.
\end{theorem}



\begin{example}
We write down explicitly all  representations corresponding to rank two motivic local systems with the above specified local monodromies. Our formulas are obtained by plugging the values of the Cayley solutions \cite[p.145]{lisovtyk} into \cite[Equation (24)]{boalchklein}\footnote{in the following, $\alpha, \beta$ correspond $r_Y, r_Z$, respectively, of \cite{lisovtyk}}. Since the fundamental group $\pi_1(X)$ is the free group generated by loops around $0,1,\lambda$, it suffices to write down three matrices $M_0, M_1, M_{\lambda}$, corresponding to the images of the three generators:
\[
M_{0}= \begin{pmatrix} 1+x_2x_3/x_1 & -x_2^2/x_1 \\  x_3^2/x_1 & 1-x_2x_3/x_1 \end{pmatrix},
M_1=\begin{pmatrix} 1 & -x_1 \\ 0 & 1\end{pmatrix}, \ 
M_{\lambda}=\begin{pmatrix} 1 & 0 \\ x_1 & 1\end{pmatrix}, \ 
\]
where 
\[
x_1=2\cos\bigg(\frac{\pi(\alpha+\beta)}{2}\bigg), \ x_2=2  \sin\bigg(\frac{\pi \alpha}{2}\bigg), \ x_3= 2 \sin\bigg(\frac{\pi \beta}{2}\bigg)
\]
for $\alpha, \beta \in \mb{Q}$. 
\end{example}

The idea of the proofs of \autoref{thm:main1} and \autoref{thm:explicit-form} is to analyze local systems satisfying $(\star)$ via Katz's \emph{middle convolution} operation \cite{katz1996rigid}. The key is to observe that local systems of geometric origin satisfying $(\star)$ may be obtained via middle convolution from certain finite monodromy local systems on $X$. Namely, if $f: E\to \mathbb{P}^1$ is the double cover of $\mathbb{P}^1$ branched at $\{0,1,\infty,\lambda\}$, then any motivic local system satisfying $(\star)$ arises via middle convolution from $f_*\mathbb{L}|_X$, where $\mathbb{L}$ is a rank one local system on $E$ with finite monodromy; moreover this construction produces all local systems of geometric origin satisfying $(\star)$. Moreover, since $f_*\mathbb{L}|_X$ itself is MCG-finite, the same is true of its middle convolution.

\begin{remark}[Comparison with the work of Yang-Zuo]\begin{enumerate}
\item As mentioned above, our main tool is Katz's middle convolution functor. Our approach gives new proofs of all the main results of \cite{zuolong}. Indeed, \cite[Theorem 1.7]{zuolong} is a corollary of \autoref{thm:explicit-form}, and \cite[Theorem 1.5]{zuolong} follows immediately from \autoref{lemma:higgsvanish}.

\item
An important invariant of a local system of geometric origin is the \emph{trace field}, i.e. the smallest field generated by the traces of elements of $\pi_1(X)$. From the description as  middle convolution of rank one torsion local systems on $E$, it is straightforward to see  that the trace fields of our local systems are $\mb{Q}(\zeta_m+\zeta_m^{-1})$ for $\zeta_m$ a primitive $m$-th root of unity, with $m$ being the order of the rank one local system. It does not seem possible to derive this information with the techniques of \cite{zuolong}; Zuo has informed us that he conjectured these trace fields in 2018; see also \cite[Appendix A]{zuolong}, where this conjecture is alluded to.

This also explains a technical point of loc.cit.: indeed, there, the authors assume that $E$ has supersingular reduction at $p$, and use crucially the fact that the trace fields of all irreducible $\SL_2$-local systems on $X$ (satisfying $(\star)$) are unramified at $p$. The latter is certainly not true if $E$ has ordinary reduction, in which case there will be rank two local systems with trace field $\mb{Q}(\zeta_{p^k}+\zeta_{p^k}^{-1})$ for any $k\geq 1$, which ramifies at $p$; correspondingly, there are pre-periodic Higgs bundles which are not periodic, in contrast with the case when $E$ has supersingular reduction (see \cite[Corollary 3.4.1]{zuolong}).
\end{enumerate}
\end{remark}

\subsection{Higgs bundles and Higgs-de Rham flow}
As before, let $D=\{0,1,\infty,\lambda\}\subset \mathbb{P}^1$, $X=\mathbb{P}^1\setminus D$, and let $\mathbb{V}$ be a local system of geometric origin on $X$ of rank $2$, satisfying $(\star)$. By e.g.~\cite[Proposition 4.1.4]{landesman2022geometric}, $(\mathbb{V}\otimes \mathscr{O}_X, \on{id}\otimes d)$ canonically extends to a filtered flat vector bundle on $\mathbb{P}^1$ with logarithmic flat connection, $$(F^1, \mathscr E, \nabla: \mathscr E\to \mathscr E\otimes \Omega^1_{\mathbb{P}^1}(\log D)),$$ where $F$ restricts to the Hodge filtration on $X$, and moreover where the residues of $\nabla$ have real parts lying in $[0,1)$.

One can show (see \autoref{prop:star-decomposition}) that if $\mathbb{V}$ satisfies $(\star)$, then $F^1\simeq \mathscr{O}_{\mathbb{P}^1}$ and $\mathscr E/F^1\simeq \mathscr{O}_{\mathbb{P}^1}(-1)$. The connection $\nabla$ yields a natural map $$\theta: F^1\to \mathscr E/F^1\otimes \Omega^1_{\mathbb{P}^1}(\log D),$$ which, being a non-zero map $\mathscr{O}_{\mathbb{P}^1}\to \mathscr{O}_{\mathbb{P}^1}(1)$, vanishes at a unique point $w(\mathbb{V})\in \mathbb{P}^1$. Moreover $w(\mathbb{V})$ determines $\mathbb{V}$ up to isomorphism (\autoref{prop:up-to-iso}) . 

The following is a slight strengthening of a conjecture of Sun-Yang-Zuo: 
\begin{conjecture}[{Slight strengthening of \cite[Conjecture 4.10]{syz}}]\label{conj:torsion}
	Let $f:E\to \mathbb{P}^1$ be the double cover branched over $D$, viewed as an elliptic curve with the identity the point over $\infty$. As $\{\mathbb{V}\}$ ranges over all local systems of geometric origin satisfying $(\star)$, the $w(\mathbb{V})$ are precisely the image under $f$ of the torsion points of $E$.
\end{conjecture}
We give a short proof of this conjecture in \autoref{prop:zuoshort}, and explain how to deduce \cite[Conjecture 4.10]{syz} in \autoref{section:syz-deduction}.

The proof has two parts. The first is to show that those $\mathbb{V}$ with $w(\mathbb{V})=f(x)$, for $x\in E$ torsion, are of geometric origin; this will follow from a more precise form of \autoref{thm:explicit-form}. 

The second is to show that these are are the only points of geometric origin. We give two approaches. The first is to again use the middle convolution, which preserves the property of geometric origin and sends $\mathbb{V}$ with $w(\mathbb{V})$ not the image of a torsion point to a local system manifestly not of geometric origin. 

The second approach relies on Sun-Yang-Zuo's theory of the Higgs-de Rham flow, as we now explain; as part of it we give a simple proof of \cite[Conjecture 4.8]{syz}. One may view $F^1\oplus \mathscr E/F^1$, with the Higgs field $$\begin{pmatrix} 0 & \theta \\ 0 & 0 \end{pmatrix}$$ as a graded nilpotent Higgs bundle. Sun-Yang-Zuo, using Ogus-Vologodsky's non-abelian Hodge theory in positive characteristic \cite{ogusvologodsky}, produce an endomorphism of the moduli of nilpotent parabolic Higgs bundles $\psi_p$ which they call the Higgs-de Rham flow. They show that for Higgs bundles associated to motivic local systems, almost all of their reductions mod $p$ are $f$-\emph{periodic} under the Higgs-de Rham flow for some fixed $f$ independent of $p$; the converse is expected to hold (see e.g. \cite{rajuperiodic} or \cite[Conjecture 1.7]{shengtorsion}). The moduli of stable nilpotent graded parabolic Higgs bundles on $(\mathbb{P}^1, D)$ whose underlying graded vector bundle is $\mathscr{O}_{\mathbb{P}^1}\oplus \mathscr{O}_{\mathbb{P}^1}(-1)$ is isomorphic to $\mathbb{P}^1$ via the map sending a Higgs bundle to the vanishing locus of its Higgs field. 

Sun-Yang-Zuo conjecture \cite[Conjecture 4.8]{syz} and Lin-Sheng-Wang prove \cite{shengtorsion}:

\begin{theorem}[{\cite[Conjecture 4.8]{syz}}, {\cite[Theorem 1.6]{shengtorsion}}]\label{thm:flow}
 Viewing $\mathbb{P}^1$ as the moduli of graded nilpotent Higgs bundles as above, the map $\psi_p: \mathbb{P}^1\to \mathbb{P}^1$ fits in the following commutative diagram:
 $$\xymatrix{
 E \ar[r]^{[p]} \ar[d]^f & E \ar[d]^f \\
 \mathbb{P}^1 \ar[r]^{\psi_p} & \mathbb{P}^1
 }$$
\end{theorem}
We give a more formal statement in \autoref{section:recollections}.

This immediately implies that any point whose reduction is $f$-periodic under almost all $\psi_p$ is torsion. We will conclude the paper by giving a simple new proof of \autoref{thm:flow} and a proof of \autoref{conj:torsion}, also via the middle convolution. Briefly, we relate the moduli space of Higgs bundles in question (which is isomorphic to $\mathbb{P}^1$) to the moduli space of rank one Higgs bundles with zero Higgs field on $E$, and observe that the middle convolution commutes with the Higgs-de Rham flow. But the Higgs-de Rham flow for Higgs bundles of rank one with zero Higgs field on  $E$ is simply $[p]$.
\subsection{Acknowledgments} We are grateful for useful conversations with Bruno Klingler, Raju Krishnamoorthy, Mao Sheng, and Kang Zuo. Lam was supported by a Dirichlet Postdoctoral Fellowship at Humboldt University. Litt was supported by the NSERC Discovery Grant, “Anabelian methods in arithmetic and algebraic geometry.”

\section{Preliminaries on convolution in the Betti and $\ell$-adic settings}\label{section:MC-prelims}
Let $k$ be a field and let $D\subset \mathbb{P}^1_k$ be a reduced effective divisor containing $\infty$. Let $X=\mathbb{P}^1\setminus D$. 
There are natural projection maps $$\pi_i: X\times X\setminus \Delta \to X,$$ for $i=1,2$ as well as a map $$m: X\times X\setminus \Delta\to \mathbb{G}_m$$ given by
 $$m: (x, y)\mapsto x-y.$$ 
 Let $$j: X\times X\setminus \Delta\hookrightarrow \mathbb{P}^1_k\times X$$ be the evident inclusion.
\begin{definition}
	Let $\chi$ be a rank one local system on $\mathbb{G}_m$, and let $\mathbb{V}$ be a local system on $X$. The \emph{middle convolution} $MC_\chi(\mathbb{V})$ is defined to be $$MC_\chi(\mathbb{V})=R^1(\pi_2)_*j_*(\pi_1^*\mathbb{V}\otimes m^*\chi).$$ 
\end{definition}
\begin{remark}
	In the $\ell$-adic or Betti setting, the middle convolution has a nice geometric interpretation. If $\mathbb{V}$ is \emph{pure of weight zero} and $\chi$ has finite order, then $MC_\chi(\mathbb{V})$ is precisely the weight one part of $R^1(\pi_2)_*(\pi_1^*\mathbb{V}\otimes m^*\chi).$
\end{remark}
 We will also make use of the following variant of middle convolution. Let $Y$ be a smooth proper curve, and let $f: Y\to \mathbb{P}^1$ be a finite, generically \'etale map, with branch divisor $D$. Let $I=f^{-1}(\infty)$. Let $\Gamma$ be the graph of $f$. Then again there are natural projections $$\pi_1: Y\times X\setminus (\Gamma\cup I\times X)\to Y,  \pi_2: Y\times X\setminus (\Gamma\cup I\times X)\to X$$ and a natural map $$m: Y\times X\setminus (\Gamma\cup I\times X)\to \mathbb{G}_m$$ $$(y,x)\mapsto f(y)-x.$$ Let $$i: Y\times X\setminus (\Gamma\cup I\times X)\hookrightarrow Y\times X$$ be the evident inclusion.
 \begin{definition}\label{defn:mcf}
	Let $\chi$ be a rank one local system on $\mathbb{G}_m$, and let $\mathbb{V}$ be a local system on $Y$. The \emph{middle convolution} $MC_\chi(f,\mathbb{V})$ is defined to be $$MC_\chi(f,\mathbb{V})=R^1(\pi_2)_*i_*(\pi_1^*\mathbb{V}\otimes m^*\chi).$$  
\end{definition}
It follows from e.g.~proper base change that:
\begin{proposition}
	There is a natural isomorphism $MC_\chi(f, \mathbb{V})\simeq MC_\chi(f_*\mathbb{V})$.
\end{proposition}
The definitions are above are convenient for computing in the Betti and $\ell$-adic settings; in the positive characteristic de Rham and Higgs setting we will use slightly different formalisms.

\section{Computing the convolution}
Suppose $D$ has even degree  with $\deg(D)\geq 4$ (with the case $\deg(D)=2$ being uninteresting) and let $f: Y\to \mathbb{P}^1$ be the double cover branched over $D$. Let $\mathbb{L}$ be a rank one local system on $Y$. Suppose $k$ has characteristic different from $2$, and let $\chi$ be the unique non-trivial rank one local system on $\mathbb{G}_m$ with $\chi^2=\text{triv}.$ We write $MC_{-1}$ instead of $MC_{\chi}$ to emphasize this particular choice of $\chi$. 

\subsection{Betti and $\ell$-adic computations}
\begin{proposition}\label{prop:localmon}
	The rank of $MC_{-1}(f, \mathbb{L})$ is $\deg(D)-2$. For $\mb{L}$ non-trivial, the local monodromies of $MC_{-1}(f, \mathbb{L})$ are given by the following. Let $J(\alpha, \ell)$ denote the Jordan block given by a  matrix of size $\ell$ and generalized eigenvalues $\alpha$.
	\begin{itemize}
	\item At a point $P\neq \infty$, there is a Jordan block $J(1,2)$, and all other Jordan blocks are $J(1,1)$.
	\item At $P=\infty$, there is a Jordan block $J(-1,2)$, and all other Jordan blocks are $J(-1,1)$.

	\end{itemize}

\end{proposition}
\begin{proof}
Given $x\in X$, the fiber $U_x$ of $\pi_2: Y\times X\setminus (\Gamma\cup I\times X)\to X$ over $x$ is isomorphic to $Y\setminus f^{-1}(\{x\cup \infty\})$, which has Euler characteristic $1-\deg(D)$. 
The restriction $\pi_1^*\mathbb{L}\otimes m^*\chi|_{U_x}$ has non-trivial monodromy at the two points of $f^{-1}(x)$ 
and trivial monodromy at $f^{-1}(\infty)$.  
Hence the Grothendieck-Ogg-Shafarevich formula yields $$\chi(j_*(\pi_1^*\mathbb{L}\otimes m^*\chi|_{U_x}))=2-\deg(D).$$ As $\pi_1^*\mathbb{L}\otimes m^*\chi|_{U_x}$ is a non-trivial rank one local system, $H^0(\pi_1^*\mathbb{L}\otimes m^*\chi|_{U_x})=H^0(\pi_1^*\mathbb{L}\otimes m^*\chi|_{U_x})=0$. Hence $MC_{-1}(f, \mathbb{L})$ has rank equal to $\deg(D)-2$.  

For the computation of local monodromies, we refer the reader to \cite[Lemma 5.1]{drpainleve}. Indeed, $MC_{-1}(f, \mb{L})$ is equivalent to the standard middle convoluiton applied to  (the restriction to $X$ of) $f_*\mb{L}$, and the latter satisfies the conditions of \cite[Lemma 5.1]{drpainleve} by our assumptions that $\mb{L}$ is non-trivial.
\end{proof}
The upshot of this proposition is that if $D=(0)+(1)+(\lambda)+(\infty)$, so $Y=E$ is an elliptic curve, and $\mathbb{L}$ is a rank one local system on $E$, then $MC_{-1}(f, \mathbb{L})$ has rank $2$ and satisfies $(\star)$. If $\mathbb{L}$ is unitary (hence underlies a $\mathbb{C}$-VHS), then $MC_{-1}(f, \mathbb{L})$ underlies a $\mathbb{C}$-VHS as well.
\subsection{de Rham computations}
In the special case when $D=(0)+(1)+(\lambda)+(\infty)$,  $Y$ is an elliptic curve with origin $e=f^{-1}(\infty)$, and we denote it by $E$ instead. Let $(\mathscr{L}, \nabla)$ be a locally free sheaf on $E$ of rank one and degree zero with integrable connection. If $k=\mathbb{C}$, the Riemann-Hilbert correspondence gives $\mathbb{L}:=\ker(\nabla)$ a rank one local system on $E$. We may write $\mathscr{L}=\mathcal{O}(e-q)$ for some $q\in E$. In this section we compute the parabolic Higgs bundle on $(\mathbb{P}^1, D)$ obtained from $(\mathscr{L},\nabla)$ via middle convolution.

Let $x\in\mathbb{P}^1\setminus D$ be a point and $x_1, x_2$ its preimages in $E$. Consider the rank one locally free sheaf with flat connection on $(E-\{x_1\}-\{x_2\})\times \{x\}$ corresponding to the character $m^*\chi$ under the Riemann-Hilbert correspondence, and  let $(\mathscr{M}_\star, \nabla)$ denote its Deligne canonical extension to $E$, where we view $\mathscr{M}_\star$ as a parabolic bundle. The degree of $\mathscr{M}_0$ is $-1$; it is $\mathscr{O}_E(-p)$ where $2p=x_1+x_2=e$. (To see this, note that $m^*\chi$ is trivialized on the double cover of $E$ branched at $x_1, x_2$, which is $\on{Spec}_E \mathscr{O}_E\oplus \mathscr{O}_E(-p),$ where the multiplication is induced by the map $\mathscr{O}_E(-2p)\to \mathscr{O}_E$ vanishing at $x_1, x_2$.)

The following is obtained by specializing \cite[Theorem 5.1.6]{llcanonical} to our setting:  \footnote{we match up the notation with loc.cit. for the reader's convenience. Let $E^{\circ}$ denote $E-E[2]$. In the notation of loc.cit., we set  $\mscr{B}\subset E^{\circ}$ to be a contractible neighborhood of $x_1$, $\mscr{C}$  the constant family $E\times \mscr{B}$, $s_{1}: \mscr{B}=E-E[2]\rightarrow \mscr{C}$ the diagonal map  and $s_2=-s_1$, and $\mb{V}$ the unitary local system  on $\mscr{C}$ whose restriction to each fiber $\mscr{C}_y$ is $\mb{L}\otimes (m^*\chi)|_{\mscr{C}_y}$} 
\begin{proposition}

The Higgs field on $MC_{-1}(f, \mathbb{L})$ at $x$ is computed via the  map $$H^0(E, \mathscr{L}\otimes \mathscr M_0\otimes \omega_E(x_1+x_2))\to H^1(E, \mathscr{L}\otimes \mathscr M_0)\otimes T^*_xX$$ given as adjoint (via Serre duality) to the multiplication map
$$H^0(E, \mathscr{L}\otimes \mathscr M_0\otimes \omega_E(x_1+x_2))\otimes H^0(E, \mathscr{L}^\vee\otimes \mathscr{M}_0^\vee\otimes \omega_E)\to H^0(\omega_E^{\otimes 2}(x_1+x_2))\to T_x^*X,$$ 
where the map $H^0(\omega_E^{\otimes 2}(x_1+x_2))\to T_x^*X$ is the one given by pullback on cotangent spaces along the map  $E\rightarrow \mc{M}_{1,2}$ given by 
\[
x_1 \mapsto (E, x_1, -x_1).
\]

\end{proposition}

\begin{lemma}\label{lemma:higgsvanish}
The Higgs bundle corresponding to $MC_{-1}(f, \mb{L})$ has Higgs field vanishing only at $f(q)$. \end{lemma}
\begin{proof}
We first show that the Higgs field vanishes at $x=f(q)$. The Higgs field factors through 
\[
H^0(\omega_E(-q+x_1+x_2)) \otimes H^0(\omega_E(q)) \rightarrow H^0(\omega_E^{\otimes 2}(x_1+x_2)).
\]
On the other hand, when $x_1=q$, we have the following commutative diagram
\begin{equation}\label{eqn:syz}
\begin{tikzcd}
H^0(\omega_E) \otimes H^0(\omega_E)  \arrow[r] \arrow[d, "\simeq"]
& H^0(\omega_E^{\otimes 2}) \arrow[d] \\
H^0(\omega_E(-q+x_1+x_2)) \otimes H^0(\omega_E(q)) \arrow[r]
& H^0(\omega_E^{\otimes 2}(x_1+x_2)).
\end{tikzcd}
\end{equation}
Indeed, the natural map $H^0(\omega_E)\rightarrow H^0(\omega(x_2))$ is an isomorphism, and the same is true of the map $H^0(\omega_E)\rightarrow H^0(\omega_E(q))$. Finally, the composition $H^0(\omega_E^{\otimes 2}) \rightarrow H^0(\omega^{\otimes 2}(x_1+x_2)) \rightarrow T^*_xX$ is zero, since this is the map induced by pullback of differentials along the composite map 
\[
E\xrightarrow{x\mapsto (E, x,-x)} \mc{M}_{1,2} \rightarrow \mc{M}_{1,1},
\]
which is constant.
\end{proof}
\section{Proofs of the main theorems in characteristic zero}
We may now give a short proof of some of our main theorems. We assume standard material about complex variations of Hodge structure; for a brief primer see \cite[\S3]{llcanonical} and \cite[\S4]{landesman2022geometric}. Before giving proofs we need some brief preliminaries.
\subsection{Preliminaries}
\begin{proposition}\label{prop:star-decomposition}
	Let $X=\mathbb{P}^1\setminus \{0,1,\infty, \lambda\}$ and let $\mathbb{V}$ be an irreducible complex local system on $X$ satisfying $(\star)$ and underlying a $\mathbb{C}$-VHS. Then the Deligne canonical extension $(\mathscr{E}, F^1, \nabla)$ of $(\mathbb{V}\otimes \mathscr{O}_X, \on{id}\otimes d)$ satisfies $\mathscr{E}\simeq \mathscr{O}_{\mathbb{P}^1}\oplus \mathscr{O}_{\mathbb{P}^1}(-1)$, with $F^1=\mathscr{O}_{\mathbb{P}^1}.$.
\end{proposition}
\begin{proof}
	From $(\star)$, the residue matrices of $\nabla$ at $\{0,1,\lambda\}$ have trace zero, and the residue matrix at $\infty$ has  trace $1$. This means that $\mathscr{E}$ and hence $F^1\oplus \mathscr{E}/\mathscr{F}^1$  has degree $-1$, by e.g.~\cite[B.3]{esnault1986logarithmic}. Let $$\theta: F^1\to \mathscr E/F^1\otimes \Omega^1_{\mathbb{P}^1}(\log D),$$ be the non-zero $\mathscr{O}$-linear map obtained from $\nabla$.   As in the introduction we consider the Higgs bundle $$F^1\oplus \mathscr{E}/\mathscr{F}^1$$ with the Higgs field $\begin{pmatrix} 0 & \theta \\ 0& 0\end{pmatrix}.$ 
	
	By Simpson's theory \cite[Theorem 8]{simpson1990harmonic}, this Higgs bundle is parabolically stable of parabolic degree zero. Parabolic stability implies the parabolic degree of $F^1$ is positive, and hence the honest degree satisfies $\deg F^1>-1$. The fact that $\theta$ is non-zero implies that $\deg F^1\leq \deg \mathscr{E}/F^1+\deg \Omega^1_{\mathbb{P}^1}(\log D)$, and hence that $2\deg F^1\leq \deg \mathscr{E}+2=1$. 
	
	Hence $\deg F^1=0$, and $\deg \mathscr{E}/F^1=-1$. Finally, we must have that the extension $$0\to F^1\to \mathscr{E}\to \mathscr{E}/F^1\to 0$$ splits as $\on{Ext}^1(\mathscr{O}_{\mathbb{P}^1}(-1), \mathscr{O}_{\mathbb{P}^1})=0.$
\end{proof}
Let $D=\{0, 1,\infty, \lambda\}$. Given $\mathbb{V}$ an irreducible local system satisfying $(\star)$ and underlying a (necessarily unique, by irreducibility) $\mathbb{C}$-VHS, the above proposition shows that we may canonically associate to $\mathbb{V}$ a nonzero map $\theta: \mathscr{O}_{\mathbb{P}^1}\to \mathscr{O}_{\mathbb{P}^1}(-1)\otimes \Omega^1_{\mathbb{P}^1}(\log D)=\mathscr{O}_{\mathbb{P}^1}(1),$ which vanishes at a unique point $w(\mathbb{V})$ of $\mathbb{P}^1$, as in the introduction.
\begin{proposition}\label{prop:up-to-iso}
	With assumptions as in \autoref{prop:star-decomposition}, the point $w(\mathbb{V})$ determines $\mathbb{V}$ up to isomorphism.
\end{proposition}
\begin{proof}
	The point $w(\mathbb{V})$ determines $\theta$ up to multiplication by a non-zero scalar, which by \autoref{prop:star-decomposition} above determines the stable parabolic Higgs bundle associated to $\mathbb{V}$ up to isomorphism. Now Simpson's theory \cite[Theorem 8]{simpson1990harmonic}  tells us that this parabolic Higgs bundle determines $\mathbb{V}$. 
\end{proof}
\subsection{Proofs}
We now begin with the proofs of our main theorems in characteristic zero.
\begin{lemma}\label{lemma:finite-order}
Let $X$ be a smooth variety. A rank one local system $\mathbb{L}$ on $X$ is of geometric origin if and only if it is torsion.
\end{lemma}
\begin{proof}
Torsion evidently implies geometric origin, as all torsion rank one local systems arise as summands of $\pi_*\mathbb{C}$ for $\pi: Y\to X$ some cyclic \'etale cover, so we prove the converse. 

Suppose $\mathbb{L}$ is rank one and of geometric origin. Then the monodromy representation of $\mathbb{L}$ is defined over the ring of integers $\mathscr{O}_K$ of some number field $K$. Moreover for each embedding $\mathscr{O}_K\hookrightarrow\mathbb{C}$, $\mathbb{L}\otimes_{\mathscr{O}_K}\mathbb{C}$ is unitary, as it underlies a polarizable $\mathbb{C}$-VHS of rank one. Hence if $\gamma\in \pi_1(X)$ is a loop, the scalar in $\mathscr{O}_K^\times$ given by the monodromy of $\gamma$ is an algebraic integer all of whose Galois conjugates have absolute value one, hence a root of unity.
\end{proof}

\begin{proposition}[\autoref{conj:torsion}]\label{prop:zuoshort}
	Suppose $\mathbb{V}$ is an irreducible rank $2$ local system on $X$ satisfying $(\star)$.  $\mathbb{V}$ is of geometric origin if and only if $\mathbb{V}$ underlies a $\mathbb{C}$-VHS with $w(\mathbb{V})=f(x)$, for $x$ a torsion point of $E$.
\end{proposition}
\begin{proof}
Let $q$ be a point of $E$. Then by Narasimhan-Seshadri \cite{narasimhan1965stable}, the unique unitary connection $\nabla$ on $\mathscr{L}=\mathscr{O}(q-e)$ has torsion monodromy if and only if $q$ itself is torsion; in particular the sheaf $\mathbb{L}$ of flat sections of $(\mathscr{L}, \nabla)$ is of geometric origin if and only if $q$ is torsion, by \autoref{lemma:finite-order}. 

Suppose $q$, and hence $\mathbb{L}$, is torsion. Now $\mathbb{V}=MC_{-1}(f, \mathbb{L})=MC_{-1}(f_*\mathbb{L})$ is evidently of geometric origin (as middle convolution preserves the property of being of geometric origin), is irreducible \cite[Theorem 2.9.8(2)]{katz1996rigid}, and $w(\mathbb{V})=q$ by \autoref{lemma:higgsvanish}. As $w(\mathbb{V})$ uniquely determines $\mathbb{V}$ by \autoref{prop:up-to-iso}, we have shown that if $w(\mathbb{V})$ is torsion, then $\mathbb{V}$ is of geometric origin.

Conversely, suppose $\mathbb{V}$ is of geometric origin and $w(\mathbb{V})=q$ is not torsion; we have that $\mathbb{V}=MC_{-1}(f, \mathbb{L})=MC_{-1}(f_*\mathbb{L})$ for $\mathbb{L}$ the sheaf of flat sections to the unitary flat line bundle $(\mathscr{O}(q-e), \nabla)$ as above, by \autoref{lemma:higgsvanish} and \autoref{prop:up-to-iso}. Hence \cite[Theorem 2.9.8(1)]{katz1996rigid} $f_*\mathbb{L}=MC_{-1}(\mathbb{V})$. So if $\mathbb{V}$ was of geometric origin, the same would be true for $f_*\mathbb{L}$, and hence for $f^*f_*\mathbb{L}=\mathbb{L}\oplus \mathbb{L}^\vee$. Hence $\mathbb{L}$ itself would be of geometric origin, hence have finite monodromy, by \autoref{lemma:finite-order}. But this contradicts our assumption that $q$ is not torsion.
\end{proof}
\begin{proof}[Proof of \autoref{thm:main1} and \autoref{thm:explicit-form}]
	\autoref{thm:main1} is immediate from what we've proven above; we may simply consider $MC_{-1}(f, \mathbb{L})$ as $\mathbb{L}$ varies over all rank one local systems on $E$ with finite monodromy. MCG-finiteness follows from the fact that the middle convolution is equivariant for the action of the mapping class group, and local systems with finite monodromy are evidently MCG-finite.
	
	To prove \autoref{thm:explicit-form}, note that the proof above tells us that each $\mathbb{V}$ of geometric origin satisfying $(\star)$ is of the form $R^1(\pi_2)_*i_*(\pi_1^*\mathbb{L}\otimes m^*\chi)$ for $\chi$ of order $2$ and $\mathbb{L}$ a finite order rank one local system on $E$, say of order $m$. But the local system $\pi_1^*\mathbb{L}\otimes m^*\chi$ on $U=E\times \mathbb{P}^1\setminus (\Gamma_f \cup E\times \{\infty\})$ is trivialized on the preimage of $U$ in the variety $Z'_m$ defined in the statement of \autoref{thm:explicit-form}, by construction. Hence its cohomology appears in the cohomology of $Z'_m \to \mathbb{P}^1$, as desired.
\end{proof}

\section{Recollections on parabolic bundles and Higgs-de Rham flows}\label{section:recollections}
For the rest of the paper we will not in general be working over $\mathbb{C}$. We prepare to state and prove a precise form of \cite[Conjecture 4.8]{syz}, stated earlier as \autoref{thm:flow}.
\subsection{Parabolic de Rham and Higgs bundles}
Let $\mc{C}/k$ be a smooth proper curve over a field $k$, and $\mc{D}=\sum \mc{D}_i \subset \mc{C}$ a reduced effective divisor of degree $l$. We refer the reader to \cite[Definition 2.3]{rajuperiodic} for the definition of parabolic de Rham and Higgs bundles; roughly this consists of bundles $(V_{\alpha})_{\alpha}$ indexed by $(\alpha_1, \cdots , \alpha_l)$ with $\alpha_i\in \mb{R}$, such that 
\begin{itemize}
\item each $V_{\alpha}$ is itself a de Rham (respectively Higgs)  bundle, and
\item they are equipped with maps of de Rham (respectively Higgs) bundles $V_{\alpha} \xhookrightarrow{} V_{\beta}$ for all $\alpha \geq \beta$ (where we write $\alpha\geq \beta$ whenever $\alpha_i\geq \beta_i$ for all $i$).
\end{itemize}
The bundles $(V_{\alpha})_{\alpha}$ are required to satisfy several more conditions, for which we refer the reader to loc.cit..

We make the following definition following \cite[Example 2.5]{rajuperiodic}
\begin{definition}\label{defn:shiftedtrivial}
Let $V$ be an arbitrary vector bundle on $\mc{C}$. For $(\alpha_1, \cdots , \alpha_l) \in \mb{Q}^{l}$,  we define the parabolic bundle $V(-\sum_{i=1}^l \mc{D}_i)$ by 
\[
V(-\sum_{i=1}^l \alpha_i \mc{D}_i)_{\beta}=V(\sum_{i=1}^l -\lceil \alpha_i +\beta_i \rceil \mc{D}_i).
\] 
\end{definition}
\begin{remark}
When   $\alpha_i=0$, the above gives  the trivial parabolic structure along $\mc{D}_i$. 
\end{remark}

The following gives a description of parabolic bundles on curves; see \cite{biswas1997parabolic} for the analogous result in characteristic zero:

\begin{proposition}[{\cite[Proposition 2.14]{rajuperiodic}}]\label{prop:descendparabolic}
Suppose  $\mc{C}'$ and $\mc{C}$ are smooth curves over $\mb{F}_q$ and $\mc{C}'\rightarrow \mc{C}$ is a $\mb{Z}/N$-covering, branched at the divisor $\mc{D} \subset \mc{C}$. Then there is an equivalence of categories between $\mb{Z}/N$-equivariant connections (resp. Higgs bundles) on $\mc{C}'$ and the category of parabolic connections on $\mc{C}$ whose parabolic structure is supported on $\mc{D}$ with weights in $\frac{1}{N}\mb{Z}$.
\end{proposition}

\subsection{Conjecture of Sun-Yang-Zuo}\label{section:stateconj}

We now recall the statement of  \cite[Conjecture 4.8]{syz}. We consider logarithmic graded semistable Higgs bundles on $(\mb{P}^1,D=0+1+\lambda + \infty)$ such that the underlying  graded vector bundle is isomorphic to $\mc{O}\oplus \mc{O}(-1)$. 
By \cite{syz}, the moduli space  $\mc{M}_{HIG}$ of such is isomorphic to  $\mb{P}^1$, by taking the unique zero of the Higgs field;  they then consider the \emph{twisted Higgs-de Rham flow} on $\mc{M}_{HIG}$, which in general depends on a lift of $(\mb{P}^1, D)$ to $W_2$, inducing a self-map $\psi_p: \mb{P}^1 \rightarrow \mb{P}^1$.

\begin{conjecture}[Sun-Yang-Zuo]\label{conj:syz}
For $p\neq 2$ and any lifting of $(\mb{P}^1, D)$ to $W_2$, the following diagram commutes
\begin{equation}\label{eqn:syz}
\begin{tikzcd}
E \arrow[r, "{[p]}"] \arrow[d, "f"]
& E \arrow[d, "f"] \\
\mb{P}^1 \arrow[r, "\psi_p"]
& \mb{P}^1
\end{tikzcd}
\end{equation}
where $f: E\to \mb{P}^1$ is the elliptic curve double cover branched over $D$, with $f^{-1}(\infty)$ being the identity.

\end{conjecture}

\subsection{Translation to parabolic Higgs bundles}\label{section:translate}
Following \cite{rajuperiodic, shengtorsion}, we translate \autoref{conj:syz} into the language of parabolic Higgs-de Rham flows, which is more natural for our purposes. Let $k$ denote an algebraic closure of $\mb{F}_p$, with $p$ different from $2$,  and let $F_k: \Spec(k) \rightarrow \Spec(k)$ denote absolute Frobenius.
\begin{definition}
\begin{enumerate} \item 
Let $HIG^{par}_{p-1, N}(\mc{C}, \mc{D})$ denote the category of parabolic logarithmic Higgs bundles on $(\mc{C}, \mc{D})$ which are nilpotent of exponent $\leq p-1$, whose parabolic structures are supported on $\mc{D}$, with weights lying in $\frac{1}{N}\mb{Z}$.

\item Let $MIC^{par}_{p-1, N}(\mc{C}, \mc{D})$ be the category of adjusted (see \cite[Definition 2.9]{rajuperiodic})  logarithmic parabolic flat bundles on $(\mc{C}, \mc{D})$, whose $p$-curvatures and nilpotent part of the residues are nilpotent of exponent $\leq p-1$, and whose parabolic structures are  supported on $\mc{D}$ with weights lying in $\frac{1}{N} \mb{Z}$. 
\end{enumerate}
\end{definition} 

From \cite[Theorem 2.10]{rajuperiodic}, we have the parabolic version of the inverse Cartier transform:
\[
C^{-1}_{par}: HIG_{p-1,N}^{par}(\mc{C}, \mc{D}) \rightarrow MIC_{p-1,N}^{par}(\mc{C}, \mc{D}). 
\]
This in general  depends on a chosen lift of $(\mc{C}, \mc{D})$ to $W_2$. By taking the associated graded of the Harder-Narasimhan filtration, we also have the functor $\Gr: MIC_{p-1,N}^{par}(\mc{C}, \mc{D}) \rightarrow HIG_{p-1,N}^{par}(\mc{C}, \mc{D}) $. 

\begin{definition}
Let $\mc{M}_{\frac{1}{2}\infty}$ denote the moduli of graded semistable parabolic Higgs bundles on $(\mb{P}^1,D)$,  with underlying the graded parabolic   bundle $(\mc{O}\oplus\mc{O}(-1))(-\frac{1}{2}\infty)$ in the notation of \autoref{defn:shiftedtrivial}.
\end{definition}

\begin{proposition}[{\cite[Proposition 2.7]{shengtorsion}}]
The functor $\Gr\circ C^{-1}_{par}$ induces a self-map on $\mc{M}_{\frac{1}{2}\infty}$. Moreover there is a natural isomorphism $\mc{M}_{HIG}\simeq \mc{M}_{\frac{1}{2}\infty}$, identifying $\psi_p$ with $\Gr\circ C^{-1}_{par}$. 
\end{proposition}

\section{Proof of Theorems in positive characteristic}

We now let $\lambda \in \mb{F}_q$ be distinct from $0$ and $1$, and denote by  $X$ the curve $\mb{P}^1\setminus \{0,1,\lambda, \infty\}$. As before, let $E\rightarrow \mb{P}^1$ be  the elliptic curve double cover branched at $0, 1, \lambda, \infty$, and let $E^{\circ}$ denote the curve $E-E[2]$.

\subsection{Middle convolution, again}\label{section:mcagain}
We now reinterpret the constructions of \autoref{section:MC-prelims} to define de Rham and Higgs analogues of the local systems $MC_{-1}(f_*\mathbb{L})$ constructed earlier, in positive characteristic.

We fix for the rest of this section a cyclic \'etale cover $\eta: \tilde{E}\rightarrow E$ of degree $m$, for some $m$ odd. We  denote by $S$ the surface $E\times \mb{P}^1$, and by $p_1, p_2$  the projection from $S$ to $E$ and $\mb{P}^1$ respectively. Let $D_S$ be the divisor $\Gamma + E\times {\infty}$ on $S$, where $\Gamma$ denotes the graph of $f: E\rightarrow \mb{P}^1$.

\begin{definition}\label{defn:surfaceszw}
Let $w: W'\rightarrow S$ be the double cover branched along $D_S$. Define the fiber product $Z'\defeq (\tilde{E}\times\mb{P}^1)\times_{E\times \mb{P}^1} W'$.

 For  $x\in X$, let $Z_x, W_x$ denote the fibers of $Z', W'$ at $x$; these fit into the following Cartesian square:
 \begin{equation}
 \begin{tikzcd}
Z_x \arrow[r] \arrow[d, "2:1"]
& W_x \arrow[d, "2:1"] \\
\tilde{E} \arrow[r,  "\eta"]
& E
\end{tikzcd}
\end{equation}
\end{definition}

\begin{notation}
Denote by  $Z_X, W_X$  the fiber products  $Z' \otimes_{\mb{P}^1} X, W' \otimes_{\mb{P}^1} X$ respectively, and  by  $J_{Z, X}, J_{W, X}$  the relative Jacobians of $Z_X, W_X$ over $X$. Also we write $z_X: Z_X\rightarrow X, w_X: W_X\rightarrow X$ for the natural maps.
 \end{notation}

\begin{proposition}\label{prop:resolve}
The pullbacks $W_{E^{\circ}}\defeq W'\times_{\mb{P}^1} E^{\circ} $, $Z_{E^{\circ}}\defeq Z'\times_{\mb{P}^1} E^{\circ} $, equipped with maps $w_{E^{\circ}}: W_{E^{\circ}}\rightarrow E^{\circ}$, $z_{E^{\circ}}: Z_{E^{\circ}}\rightarrow E^{\circ}$, may be compactified to semistable families  $w_E: W_E\rightarrow E$, $z_E: Z_E\rightarrow  E$ over $E$. Furthermore we may choose $Z_E$ and $W_E$  such that 
\begin{itemize} 
\item the natural $(\mb{Z}/m\times \mb{Z}/2)\times \mb{Z}/2$-action \footnote{the $\mb{Z}/m\times \mb{Z}/2$-action comes from the construction of  $Z'$ as a fiber product in Definition~\ref{defn:surfaceszw}, whereas   the  $\mb{Z}/2$-action one acting on $Z_{E^{\circ}}=Z\times_{\mb{P}^1} E^{\circ}$ via the $-1$-map on the second factor} on $Z_{E^{\circ}}$ extends,
\item the natural $\mb{Z}/2\times \mb{Z}/2$-action on $W_{E^{\circ}}$ extends, and 
\item there is a commutative diagram 
\begin{equation}
\begin{tikzcd}[column sep=scriptsize]
Z_E \arrow[dr, "z_E" '] \arrow[rr, "\theta"]{}
& & W_E \arrow[dl, "w_E"] \\
& E 
\end{tikzcd}
\end{equation}
extending the natural one with $Z_{E^{\circ}}, W_{E^{\circ}}, E^{\circ}$, and such that $\theta$ is equivariant for the first $\mb{Z}/2$-action, and the entire diagram is equivariant for the second $\mb{Z}/2$-action.
\end{itemize}
\end{proposition}
\begin{proof}[Proof sketch]
	The desired compactifications are obtained by normalizing $W', Z'$ in the function fields of $W_{E^\circ}, Z_{E^\circ}$. The three bullets are clear. One may prove semistability in a number of ways, e.g.~by computing with local models; we now sketch a proof of semistability by analyzing the monodromies of the local systems $R^1w_{E^\circ, *}\Qlbar, R^1z_{E^\circ, *}\Qlbar$ about the points of $E[2]$. As a curve has semistable reduction if and only if its Jacobian does, it suffices to show, by the N\'eron-Ogg-Shafarevich criterion, that these local monodromies are unipotent.
	
Recall that we have the map $z_X: Z_X\rightarrow X$. It suffices to show (as $E\to \mathbb{P}^1$ is ramified to order $2$ at $0,1,\infty, \lambda$) that the local monodromies of $R^1z_{X, *}\Qlbar$ around all punctures have generalized eigenvalues $1$ or $-1$. This is true for $R^1z_{X, *}\Qlbar/R^1w_{X, *}\Qlbar$, since this is a sum of middle convolutions as in \autoref{thm:explicit-form}, so it remains to show the same for $R^1w_{X, *}\Qlbar$. 

Denote the map $E^{\circ}\rightarrow X$ by  $\pi^{\circ}$, and let $\tau$ denote the non-trivial summand of $R^1\pi^{\circ}_*\Qlbar$. Then $R^1w_{X*}\Qlbar$ is a direct sum of the  local system associated to the constant family $E\times X$, and $MC_{-1}(\tau)$. For $MC_{-1}(\tau)$, the local monodromies at the punctures satisfy $(\star)$, again by \cite[Lemma 5.1]{drpainleve}.
\end{proof}

\begin{definition}\label{defn:mchiggs}
\begin{enumerate}
\item
For each character $\chi: \mb{Z}/m \rightarrow \fpbar$, we define $\widetilde{MC}_{-1}^{dR}(\chi)$ to be the $\chi\otimes (-1)$-isotypic\footnote{here $\chi\otimes (-1)$ denotes the  character of $(\mb{Z}/m\times \mb{Z}/2)$, the first factor of the group $(\mb{Z}/m\times \mb{Z}/2)\times \mb{Z}/2$ from \autoref{prop:resolve},   given by tensoring $\chi$ with the $-1$-character of $\mb{Z}/2$} part of the relative log-de Rham cohomology of $(Z_E, z_E^{-1}(E[2])$  over $(E, E[2])$. 

\item
Moreover,  the additional $\mb{Z}/2$-action in Proposition~\ref{prop:resolve} implies that  each  $\widetilde{MC}_{-1}^{dR}(\chi)$ has a $\mb{Z}/2$-equivariant structure with respect to the covering $E\rightarrow \mb{P}^1$, and so descends to a  parabolic logarithmic de Rham bundle on $(\mb{P}^1, D)$ by \autoref{prop:descendparabolic}, which we denote by $MC_{-1}^{dR}(\chi)$; moreover the latter is equipped with a Hodge filtration descended from $E$. Finally, let $MC_{-1}(\chi)$ denote the parabolic, graded, logarithmic, Higgs bundle on $(\mb{P}^1, D)$ given by the associated graded of $MC_{-1}^{dR}(\chi)$ with respect to the Hodge filtration.
\end{enumerate}
\end{definition}

For each $x\in \mb{P}^1$ with $x\neq 0,1,\lambda, \infty$, let $W_x$ be the fiber of $W'$ above $x$; $W_x$ is a smooth curve and is moreover  equipped with a $\mb{Z}/2$-covering map $h_{W, x}: W_x\rightarrow E_x=E\times \{x\}$.

By \autoref{prop:descendparabolic}, the $\mb{Z}/2$-equivariant bundle with connection  $\mc{O}_{W_x}$ descends to a parabolic logarithmic connection $J_x$ on $E_x$. Let $L_{\chi, x}$ denote the $\chi$-isotypic component of $\eta_*\mc{O}_{\tilde{E}}$. Note that the isomorphism classes of $E_x$ and $L_{\chi,x}$ are  independent of the choice of $x$, and we simply denote them by $E, L_{\chi}$ respectively. The line bundle $L_{\chi}$ has degree zero, and so is isomorphic to $\mc{O}_E(e-\tilde{y})$ for a unique $\tilde{y}\in E$; we say that $y\defeq \pi(\tilde{y})$ is the point corresponding to $L_{\chi}$.

\begin{proposition}\label{prop:higgszero2}
The fiber of $MC_{-1}(\chi)$ at $x$ is given by 
\begin{equation}\label{eqn:mcfiber}
H^1(L_{\chi}\otimes J_x) \oplus H^0(L_{\chi}\otimes J_x \otimes \Omega^1_{E}(f^{-1}(x))),
\end{equation}
with the first (respectively second) factor being the degree one (zero) piece. 

The Higgs field vanishes only at the point $y$ corresponding to $L_{\chi}$. Moreover the graded vector bundle underlying $MC_{-1}(\chi)$ is isomorphic to $\mc{O}\oplus \mc{O}(-1)$.
\end{proposition}

\begin{proof}
Since the maps $Z_X, W_X \rightarrow X$ are  smooth, the restriction $MC^{dR}_{-1}(\chi)|_X$ is given by the quotient of the relative de Rham cohomologies of $Z_X$ and $W_X$ over $X$. Therefore the fiber $MC_{-1}^{dR}(\chi)|_x$ is canonically identified with the  $\chi \otimes (-1)$-isotypic component of $H^1_{dR}(Z_x)$. 

 Let $h_{Z, x}: Z_x\rightarrow E$ be the natural map. To take the $\chi\otimes (-1)$ isotypic component, we can first take the $\chi\otimes (-1)$-component of $h_{Z,x, *}^{dR}\mc{O}_{Z_x}$; the latter is a logarithmic connection on the line bundle $L_{\chi}\otimes J_x$, with poles along $f^{-1}(x)$. The fiber of $MC_{-1}^{dR}(\chi)$ at $x$ is then given by taking  $\mb{H}^1$ of the associated de Rham complex 
\[
L_{\chi}\otimes J_x \rightarrow L_{\chi}\otimes J_x\otimes \Omega^1_{E}(f^{-1}(x)),
\]  
with the Hodge filtration given by the naive filtration. Therefore the fiber of the Higgs cohomology is given by \eqref{eqn:mcfiber}. 

Denote the graded vector bundle underlying  $MC_{-1}(\chi)$ by $\mc{F}$;  we now prove that  $\mc{F}\simeq \mc{O}\oplus \mc{O}(-1)$ as a graded bundle.

Suppose $\lambda_W\in W(\mb{F}_q)$ is a lift of $\lambda$, so that the divisor $D$ lifts canonically to a divisor $D_W\subset \mb{P}^1_W$, and we also have the pair $(E_W, E_W[2])$ lifting $(E, E[2])$. Fixing a lift $\chi_W: \mb{Z}/m \rightarrow W(\fpbar)^{\times}$ of $\chi$, we can then define a graded logarithmic Higgs bundle $MC_{-1, \mb{P}^1_W}(\chi_W)$ on $(\mb{P}^1_W, D_W)$ exactly as in \autoref{defn:mchiggs}, which lifts $MC_{-1}(\chi)$.   



Let us write the generic fiber $MC_{-1, \mb{P}^1_W}(\chi_W)[1/p]$ as $M\oplus N $, with $M$ (respectively $N$) being the degree one (resp. zero) piece for the Hodge grading; this is a stable Higgs bundle   of degree zero on $\mb{P}^1_{W[1/p]}$, with non-vanishing Higgs field. Under the usual Simpson correspondence, there is some $\mb{L}$ such that this Higgs bundle corresponds to the local system  $MC_{-1}(f, \mb{L})$ in the notation of \autoref{defn:mcf}, which has  unipotent monodromy around $0, 1, \lambda_W$ by \autoref{prop:localmon}; therefore the parabolic weights are zero at these points. This forces $\deg(M)=0,$ $\deg(N)=-1$ by \autoref{prop:star-decomposition}; since Chern classes are invariant under specialization, the same is true of $\mc{F}$, as required.


Finally,  let $L_{\chi_W}$ denote the $\chi_W$-isotypic component of $\eta_*\mc{O}_{\tilde{E}_W}$, which is a line bundle lifting $L_{\chi}$. By \autoref{lemma:higgsvanish},  $\theta$ vanishes at the point $\tilde{x} \in \mb{P}^1_{W[1/p]}$ corresponding  to $L_{\chi_W}$, and hence the Higgs field of $MC_{-1}(\chi)$ vanishes at the point $x$ corresponding to $L_{\chi}$, as claimed.  
\end{proof}

\begin{corollary}\label{lemma:higgspoint}
For all choices of $\eta: \tilde{E}\rightarrow  E$ and $\chi$,  $MC_{-1}(\chi)$ corresponds to a point in $\mc{M}_{HIG} \simeq \mc{M}_{\frac{1}{2}\infty}$, where these spaces are defined in  \ref{section:stateconj}, \ref{section:translate}.
\end{corollary}

\subsection{}

Since $z_E: Z_E\rightarrow E$ and $w_E: W_E\rightarrow E$ are semistable families of curves, we may define their relative log Higgs cohomologies (i.e. the associated graded of relative  log de Rham cohomology with respect to the Hodge filtration), and  by another application of Proposition~\ref{prop:descendparabolic} we may descend to obtain parabolic logarithmic Higgs bundles on $\mb{P}^1$, which we denote by $R^1z_*^{HIG}\mc{O}_Z$ and $R^1w_*^{HIG}\mc{O}_W$ respectively. Let $(\mc{E}, \theta)$ denote the quotient Higgs bundle $R^1z_*^{HIG}\mc{O}_Z/R^1w_*^{HIG}\mc{O}_W$. Similarly we have the logarithmic flat bundles $R^1z_*^{dR}\mc{O}_Z, R^1w_*^{dR}\mc{O}_W, (\mc{E}_{dR}, \nabla)\defeq R^1z_*^{dR}\mc{O}_Z/R^1w_*^{dR}\mc{O}_W$.



\begin{proposition}\label{prop:periodic}
$C^{-1} (\mc{E}, \theta)|_X \simeq (\mc{E}_{dR}, \nabla)|_X$, and therefore $\Gr \circ C^{-1} (\mc{E}, \theta)|_X \simeq (\mc{E}, \theta)|_X$. Here $C^{-1}$ denotes the usual Cartier transform, and $\Gr$ the associated graded with respect to the Hodge filtration on de Rham cohomology. 
\end{proposition}

\begin{proof}
We have $(\mc{E}, \theta)|_X\simeq R^1z_{X*}^{HIG} \mc{O}_{Z_X}/R^1w_{X*}^{HIG} \mc{O}_{W_X}$. The first statement now follows from \cite[Theorem 3.4]{ogusvologodsky} since we have lifts of $Z_X, W_X$ to $W_2$, and the second  is immediate.
\end{proof}

\begin{proof}[Proof of Conjecture~\ref{conj:syz}]
Let $g: \mb{P}^1\rightarrow \mb{P}^1$ be the map making \eqref{eqn:syz} commute. Then $g$ has degree $p^2$, and the same is true of $\Gr\circ C^{-1}_{par}$  by \cite[\S~4.3]{syz}; therefore it suffices to show that $\Gr\circ C^{-1}_{par}$ and $g$ agree on infinitely many closed points. For any $m\geq 1$ odd, let $\tilde{E}\rightarrow E$ be a non-trivial \'etale $\mb{Z}/m$-covering, and for any character $\chi: \mb{Z}/m \rightarrow \fpbar$ we have the Higgs bundles  $MC_{-1}(\chi)$ on $(\mb{P}^1, D)$. By \autoref{lemma:higgspoint}, by tensoring with a rank one Higgs bundle if necessary, these Higgs bundles correspond to points in  $\mc{M}_{HIG}\simeq \mb{P}^1$ and we will show that $g$ and $\Gr\circ C^{-1}_{par}$ agree on all of these points.

Recall that $F_k: k\rightarrow k$ denotes absolute Frobenius.  The Higgs-de Rham functor $\Gr\circ C^{-1}$ is $F_k$-linear, and by Proposition~\ref{prop:periodic} $\Gr\circ C^{-1}(\mc{E}, \theta)|_X \simeq (\mc{E}, \theta)|_X$. Now  $MC_{-1}(\chi)|_X$ is the $\chi$-isotypic component for the $\mb{Z}/m$-action on $(\mc{E}, \theta)|_X$, and therefore  $\Gr\circ C^{-1}(MC_{-1}(\chi))|_X$ is the $\chi^p$-isotypic component. On the other hand, by definition of the $MC_{-1}(\chi)$'s  the $\chi^p$-isotypic component is given by $MC_{-1}(\chi^p)|_X$. Therefore $\Gr \circ C^{-1}(MC_{-1}(\chi))|_X\simeq MC_{-1}(\chi^p)|_X$, and hence $\Gr \circ C_{par}^{-1}(MC_{-1}(\chi))$ and $MC_{-1}(\chi^p)$ are isomorphic up to tensoring with a rank one Higgs bundle with trivial Higgs field \footnote{note that $\Gr$, which is a priori the Hodge filtration restricted to $MC_{-1}(\chi)$, is also the Harder-Narasimhan filtration, which is what  Sun-Yang-Zuo use to formulate their conjecture}.   By the explicit description of $MC_{-1}(\chi)$ given in Proposition~\ref{prop:higgszero2}, in particular the claim about the vanishing locus of the Higgs field,  we may conclude.
\end{proof}

\section{Some loose ends}
\subsection{}
For each $\lambda \in \mb{F}_q$, suppose we have a  rank two $\Qlbar$-local system $\mb{V}$ on $X=\mb{P}^1-\{0, 1, \lambda, \infty\}$ satisfying $(\star)$. Let $\lambda_W\in W(\mb{F}_q)$ be a lift of $\lambda$, and let $X_W=\mb{P}^1-\{0, 1, \lambda_W, \infty\}$ be the corresponding lift of $X$, and $\tilde{f}$ the lift of $f$. By definition $\mb{V}$ corresponds to a representation of the tame fundamental group $\pi_1^t(X)$; the specialization isomorphism $\pi_1^t(X_W)\simeq \pi_1^t(X)$ gives a local system on $X_W$, which we denote by $\mb{V}_W$.
\begin{proposition}
The local system $\mb{V}_W$ is motivic, i.e.  there exists an abelian scheme $\pi: \mc{A}\rightarrow X_W$ such that $\mb{V}$ appears as a direct summand of $R^1\pi_*\Qlbar$.
\end{proposition}

\begin{proof}
Note that, by using the Weil restriction of abelian schemes,  it suffices to show that the pullback of $\mb{V}_W$ to $X_W\otimes \Spec(W(\mb{F}_{q^2}))$ appears as a 
direct summand of an abelian scheme.  This now follows by applying  $MC_{-1}$ to $\mb{V}$. Indeed, by  \cite[Lemma 5.1]{drpainleve}, $MC_{-1}(\mb{V})$ has rank two, and the local monodromies at $0,1, \lambda, \infty$ are each semisimple with eigenvalues $1$ and $-1$. As in the rest of the paper, let  $f: E\rightarrow \mb{P}^1$ be the elliptic curve branched at $0, 1, \lambda, \infty$, and let $f^{\circ}: E^{\circ}=E-E[2]\rightarrow X$ denote the restriction to $X$.   Therefore  $f^{\circ *}MC_{-1}(\mb{V})$ extends to  a local system on $E$, and let us denote this by $\mb{W}$. Then either $\mb{W}\simeq \mb{L}\oplus [-1]^*\mb{L}$ for a rank one local system $\mb{L}$ on  $E$, or it is of the form $h_*\mb{L}$ for a rank one local system $\mb{L}$ on $E_{\mb{F}_{q^2}}$, and $h$ denotes the map   $E_{\mb{F}_{q^2}}\rightarrow E$. 
Now let $E_W\rightarrow \mb{P}^1$ denote the elliptic curve branched at $0,1, \lambda_W, \infty$. In the first case above,  $\mb{L}$ lifts canonically to a rank one local system $\widetilde{\mb{L}}$ on $E_W$, and we have $MC_{-1}(\tilde{f}, \widetilde{\mb{L}}) \simeq \mb{V}_W$, and therefore the latter appears in an  abelian scheme. In the second case, we similarly have that $\mb{V}_{W(\mb{F}_{q^2})}$, the pullback of $\mb{V}_W$ to $X_W\otimes W(\mb{F}_{q^2})$, is isomorphic to $MC_{-1}(\tilde{f}, \widetilde{\mb{L}})$, and again we may conclude. 
 \end{proof}

\subsection{Another conjecture of Sun-Yang-Zuo}\label{section:syz-deduction}
\begin{proof}[Proof of {\cite[Conjecture 4.10]{syz}}]
The only if direction is immediate from Conjecture~\ref{conj:syz}, proven above, or from \autoref{prop:zuoshort}. We now check the if direction.

Suppose we have $(E, \theta) \in \mc{M}_{HIG}$ with $(\theta)_0=f(x)$ for some torsion point $x$ of order prime to $p$. Then we have shown that $(E, \theta)$ appears in a family of abelian varieties which lifts to $W(\mb{F}_q)$, and therefore lifts to  a periodic Higgs bundle in $\mc{M}_{W(\mb{F}_q)}$, by \cite[Theorem 1.4]{zuosemistable}. 
\end{proof}

\bibliographystyle{alpha}
\bibliography{bibliography-middle-convolution}

\end{document}